\documentclass[journal]{IEEEtran}
\usepackage{amsmath, amsthm, amssymb,mathtools}
\usepackage{latexsym}
\usepackage{graphicx}
\usepackage{verbatim}
\usepackage{mathrsfs}
\usepackage{float}
\usepackage[lofdepth, lotdepth]{subfig}
\usepackage{array}
\usepackage{url}
\usepackage{algorithm}
\usepackage[noend]{algorithmic}
\usepackage{cite}
\usepackage[table]{xcolor}
\usepackage{enumerate}
\usepackage{stmaryrd}
\usepackage{multirow}
\usepackage{soul}
\usepackage[normalem]{ulem}

\theoremstyle{plain}
\newtheorem{theorem}{Theorem}
\newtheorem{mtheorem}{Main Theorem}
\newtheorem{lemma}{Lemma}

\newtheorem{corollary}{Corollary}
\newtheorem{proposition}{Proposition}

\theoremstyle{definition}
\newtheorem{definition}{Definition}
\newtheorem{example}{Example}

\newcommand{\A}{{\mathcal A}}
\newcommand{\B}{{\mathcal B}}
\newcommand{\C}{{\mathcal C}}

\newcommand{\bC}{{\boldsymbol C}}

\newcommand{\bz}{{\Bbb Z}}
\newcommand{\su}{{\sf u}}
\newcommand{\sv}{{\sf v}}

\newcommand{\vv}{{\overline{v}}}
\newcommand{\vw}{{\overline{w}}}

\begin{document}
\onecolumn
\pagestyle{empty}

\title{Linear Size Constant-Composition Codes Meeting the Johnson Bound}

\author{Yeow~Meng~Chee,~{\it Senior~Member,~IEEE},~and~Xiande~Zhang
        \thanks{Y. M. Chee ({\tt ymchee@ntu.edu.sg})
        is with Division of Mathematical Sciences, School of
Physical and Mathematical Sciences, Nanyang Technological
University, Singapore 637371.}

\thanks{X. Zhang ({\tt drzhangx@ustc.edu.cn}) is with School of Mathematical Sciences,
University of Science and Technology of China, Hefei, 230026, Anhui, China. Part of this work was done while X. Zhang was with the Division of Mathematical
Sciences, School of Physical and Mathematical Sciences, Nanyang
Technological University, 21 Nanyang Link, Singapore 637371. The research of X. Zhang is supported by NSFC under grant 11301503.}
\thanks{Copyright (c) 2012 IEEE. Personal use of this material is permitted.  }
}


\maketitle

\begin{abstract}
The Johnson-type upper bound on the maximum size of a code of length $n$, distance $d=2w-1$ and constant composition $\vw$ is $\lfloor\dfrac{n}{w_1}\rfloor$, where $w$ is the total weight and $w_1$ is the largest component of $\vw$.
Recently, Chee {\em et al.} proved that this upper bound can be achieved for all constant-composition codes  of
sufficiently large lengths. Let $N_{ccc}(\vw)$ be the smallest such length. The determination of $N_{ccc}(\vw)$ is trivial for binary codes. This paper provides a lower bound on  $N_{ccc}(\vw)$, which is shown to be tight for all ternary and quaternary codes by giving new combinatorial constructions. Consequently, by refining method, we determine the values of $N_{ccc}(\vw)$ for all $q$-ary constant-composition codes provided that $3w_1\geq w$ with finite possible exceptions.
\end{abstract}


\section{Introduction}

Constant-composition codes have attracted a lot attention \cite{Svanstrom:2000,Cheeetal:2010a,Svanstrometal:2002,ChuColbournDukes:2006,ding2005combinatorial,luo2003constant,ding2005family,bounds2005algebraic,DingYin:2006,chee2007pbd,bogdanova2003enumeration,ding2008optimal,Huczynska:2010,CGL:2008,yan2009class,gao2011optimal,yin2008new,ding2008construction,zhu2012quaternary,wei2015optimal,chee2014decompositions} in recent years due to their vast applications, such as in determining the zero error decision feedback
capacity of discrete memoryless channels \cite{moulin2012log,scarlett2014refinements}, multiple-access
communications \cite{dyachkov1984random,scarlett2015second}, spherical codes for modulation \cite{ericson1995spherical}, DNA
codes \cite{king2003bounds,chee2008improved}, powerline communications \cite{chu2004constructions,colbourn2004permutation}, and frequency
hopping \cite{ChuColbournDukes:2006}.

Although constant-composition codes have been used since the early 1980s to bound error and erasure probabilities in decision
feedback channels \cite{csiszar2011information}, their systematic study only began in late
1990s with Svanstr\"om \cite{Svanstrom:1999b}. Nowadays, the problem of determining
the maximum size of a constant-composition code constitutes a central problem in their study due to their close relations to combinatorial design theory \cite{Cheeetal:2010a,Svanstrometal:2002,ChuColbournDukes:2006,ding2005combinatorial,ding2005family,DingYin:2006,chee2007pbd,CGL:2008,yan2009class,gao2011optimal,yin2008new,zhu2012quaternary,wei2015optimal}.

For integers $m\leq n$,
the set of integers $\{m,m+1,\ldots,n\}$ is denoted by $[m,n]$. When $m=1$, the
set $[1,n]$ is further abbreviated to $[n]$. If $m>n$, then $[m,n]$ is defined to be empty.
The ring $\bz/n\bz$ is denoted by $\bz_n$. 
For finite sets $R$ and $X$, $R^X$ denotes the set of vectors
of length $|X|$, where each component of a vector $\su\in R^X$ has
value in $R$ and is indexed by an element of $X$, that is,
$\su=(\su_x)_{x\in X}$, and $\su_x\in R$ for each $x\in X$.

A {\em $q$-ary code of length $n$} is a set $\C\subseteq \bz_q^X$,
for some $X$ of size $n$. The elements of $\C$ are called {\em
codewords}. The {\em support} of a vector $\su\in
\bz_q^X$ is
${\rm supp}(\su)=\{x\in X : \su_x \neq 0\}$. The {\em Hamming weight} of a
vector $\su\in \bz_q^X$ is defined as $\|\su\|=|{\rm supp}(\su)|$. The distance induced by this weight is the
{\em Hamming distance}, denoted by $d_H(\cdot,\cdot)$, so that $d_H(\su,\sv) =
\|\su-\sv\|$, for $\su,\sv\in \bz_q^X$.

A code $\C$ is said to have {\em distance} $d$ if
$d_H(\su,\sv)\geq d$ for all distinct $\su,\sv\in \C$. The {\em composition} of a vector $\su\in \bz_q^X$ is the tuple $\vw=\llbracket w_1,\ldots,w_{q-1} \rrbracket$, where $w_i=|\{x\in X: \su_x=i\}|$, $i\in \bz_q\setminus\{0\}$. A code $\C$ is said to have
{\em constant weight} $w$ if every codeword in $ \C$ has weight $w$, and have {\em constant composition} $\vw$ if every codeword  has composition $\vw$.  Hence, every constant-composition code is a constant-weight code. In this paper, attention is restricted to constant-composition codes. For constant-weight codes, interested readers are referred to \cite{SHP2006EJC}.

A $q$-ary code of
length $n$, distance $d$, and constant composition $\vw$ is denoted an
$(n,d,\vw)_q$-code.  The maximum size of an
$(n,d,\vw)_q$-code is denoted $A_q(n,d,\vw)$, and an $(n,d,\vw)_q$-code
attaining the maximum  size is said to be {\em optimal}. In an $(n,d,\vw)_q$-code, reordering the components of $\vw$ or deleting zero components of $\vw$ will not affect the distance and composition properties. Hence, through out this paper, when we talk about a composition $\vw=\llbracket w_1,w_2\ldots,w_{q-1} \rrbracket$, we always assume that all components are positive and listed in non-increasing order, that is, $w_1\geq w_2\geq \cdots \geq w_{q-1}\geq 1$. For succinctness, define the total weight $w:=\sum_{i=1}^{q-1}w_i$.

The  Johnson-type bound of Svanstr\"{o}m  for ternary constant-composition codes \cite{Svanstrom:2000}
could be easily extended to the following (see also \cite{ChuColbournDukes:2006}).

\begin{proposition}\label{johnsonbound} (Johnson Bound):
\begin{align*}
A_q(n,d,\llbracket w_1,&w_2\ldots,w_{q-1} \rrbracket)\leq\\& \left\lfloor \frac{n}{w_1}
A_q(n-1,d,\llbracket w_1-1,w_2,\ldots,w_{q-1} \rrbracket)\right\rfloor.
\end{align*}
\end{proposition}

\begin{definition} Let $q>q'$ be two positive integers. A composition $\vw=\llbracket w_1,\ldots,w_{q} \rrbracket$  is a {\em refinement} of $\vv=\llbracket v_1,\ldots,v_{q'} \rrbracket$ if there exist pairwise disjoint sets $S_1,\ldots,S_{q'}\subset [q]$ satisfying $\cup_{j\in [{q'}]}S_j=[q]$, such that $\sum_{i\in S_j}w_i=v_j$ for each $j\in [{q'}]$.
\end{definition}

Chu {\em et al.} \cite{ChuColbournDukes:2006} made the following observation.

\begin{lemma}\label{refine} If $\vw$ is a refinement of $\vv$, then $A_q(n,d,\vw)\geq A_q(n,d,\vv)$.
\end{lemma}

In \cite{Cheeetal:2010a}, Chee {\em et al.} showed that $A_q(n,d,\vw)=O(n)$ if and only if $d\geq 2w-1$. For $d\geq 2w$, it is trivial to determine values of $A_q(n,d,\vw)$. For $d=2w-1$, \[A_q(n,2w-1,\vw)\leq \left\lfloor\dfrac{n}{w_1}\right\rfloor\] for all $\vw$ by Proposition~\ref{johnsonbound}. When $q=2$, we know that $A_2(n,2w-1,\llbracket w_1 \rrbracket) =\lfloor\dfrac{n}{w_1}\rfloor$, trivially. When $q=3$, the values of $A_3(n,2w-1,\vw)$ has been completely determined by Svanstr\"{o}m {\em et al.} \cite{Svanstrometal:2002}.  Besides this, the following asymptotic statement  was proved in \cite{Cheeetal:2010a}.

\begin{theorem}[Chee {\em et al.} \cite{Cheeetal:2010a}]\label{asy} Let $\vw=\llbracket w_1,w_2\ldots,w_{q-1} \rrbracket$. Then $A_q(n,2w-1,\vw)=
\lfloor\dfrac{n}{w_1}\rfloor$ for all sufficiently large $n$.
\end{theorem}
\subsection{Problem Status and Contribution}

In Theorem~\ref{asy}, the hypothesis that $n$ is sufficiently large must be satisfied. But how large must $n$ be? More precisely, for a composition $\vw=\llbracket w_1,w_2\ldots,w_{q-1} \rrbracket$,  let
\begin{align*}
&N_{ccc}(\vw)=\min\{n_0\in {\Bbb N}:A_q(n,2w-1,\vw)=
\left\lfloor\dfrac{n}{w_1}\right\rfloor \text{ for all }n\geq n_0\},
\end{align*}
which was first defined in  \cite{Cheeetal:2010a}. For binary codes, it is trivial that $N_{ccc}(\llbracket w_1\rrbracket )=1$. Explicit bounds on $N_{ccc}(\vw)$ for general $\vw$ were given in \cite{Cheeetal:2010a}.

\begin{proposition}
\label{bound} For any composition $\vw$, we have \[w^2-w_1(w-1)\leq N_{ccc}(\vw)\leq 4w_1(w-1)^2+1.\]
\end{proposition}

The upper and lower bounds on  $N_{ccc}(\vw)$ in Proposition~\ref{bound} differ approximately by a factor of $4w_1$. Our interest in this paper is in determining the exact values of $N_{ccc}(\vw)$. In fact, a stronger lower bound of $N_{ccc}(\vw)$ is established in Section~\ref{lbound},  and proved to be tight for ternary constant-composition codes. In Section~\ref{combcons}, we provide a general combinatorial construction for optimal linear size constant composition codes. Based on this construction, Sections~\ref{ngeql} and \ref{nleql} serve to prove that our new lower bound of  $N_{ccc}(\vw)$ is also tight for quaternary constant-composition codes. Finally, by refining and lengthening techniques, we determine the values of $N_{ccc}(\vw)$ for all $\vw$ provided that $3w_1\geq w$. Our main result is summarized as below.
\begin{mtheorem}\label{main} Given a composition $\vw$ with at least two components. Let $\lambda=\lceil\frac{w}{w_1}\rceil$ and $s=\lambda w_1-w$. Then \[N_{ccc}(\vw)\geq \lambda(\lambda-1)w_1^2-2(\lambda-1)sw_1+w_1-\lfloor\frac{2s}{\lambda}\rfloor.\]
In particular, equality holds for all $\vw$ provided that $3w_1\geq w$, and $\vw$ is not a refinement of any composition in $\{\llbracket 4,4,2 \rrbracket, \llbracket 4,3,3 \rrbracket, \llbracket 5,5,3 \rrbracket, \llbracket 5,4,4 \rrbracket\}$.
\end{mtheorem}
Previously, exact values of $N_{ccc}(\vw)$ were known only for binary codes or for compositions $\vw$ with total weight at most six.

\section{Lower bounds}\label{lbound}
In this section, we prove the lower bound of $N_{ccc}(\vw)$ in Main Theorem~\ref{main}.  Chee {\em et al.} \cite{Cheeetal:2010a} showed that the following two conditions are necessary and sufficient
for a $q$-ary code $\C$ of constant weight $w$ to have distance $2w-1$:
\begin{description}
\item[(C1)] for any distinct $\su,\sv\in\C$, $|{\rm supp}(\su) \cap {\rm supp}(\sv)|\leq 1$, and
\item[(C2)] for any distinct $\su,\sv\in\C$, if $x\in {\rm supp}(\su) \cap {\rm supp}(\sv)$, then $\su_x\neq \sv_x$.
\end{description}
The idea of deducing our lower bound  is based on the above two conditions, which have been used in \cite{Cheeetal:2010a} to obtain the lower bound in Proposition~\ref{bound}.

Let $\C=\{\su^{(1)},\ldots,\su^{(|\C|)}\}$ be an $(n, 2w-1, \vw)_q$-code. Then, $\C$  can be regarded as an $|\C|\times n$ matrix $\bC$, whose $j$th row is $\su^{(j)}$, $j\in [{|\C|}]$. Let $N_i$ be the number of nonzero entries in column $i$ of $\bC$, $i\in [n]$. Then,
\begin{align}\label{nonzeros}
\sum_{i\in [{n}]}N_i=|\C|w.
\end{align}

 In each column of $\bC$, we associate each pair of distinct nonzero entries with the pair of rows that contain these entries. There are $N_i \choose 2$ such pairs of nonzero entries in column $i$ of $\bC$. Therefore, there are $\sum_{i\in [{n}]}{N_i \choose 2}$ such pairs in all columns of $\bC$. Since there are no pairs of distinct codewords in $\C$ whose support intersect in two elements, the $\sum_{i\in [{n}]}{N_i \choose 2}$ pairs of rows associated with the $\sum_{i\in [{n}]}{N_i \choose 2}$ pairs of distinct nonzero entries are also all distinct. Hence,
\begin{align}\label{pairs}
\sum_{i\in [{n}]}{N_i \choose 2}\leq {|\C| \choose 2}.
\end{align}
We will use Eq.~(\ref{pairs}) to obtain our lower bound on $N_{ccc}(\vw)$. Given a composition $\vw$, let $\lambda:=\lceil\frac{w}{w_1}\rceil$ and $s:=\lambda w_1-w$. Since $q\geq 3$,  we have $\lambda\geq 2$ and $0\leq s< w_1$.

We first deal with the case when $w_1|n$. Let $n=Mw_1$ and $|\C|=M$.  It is easy to show that the left hand side of Eq.~(\ref{pairs}), $\sum_{i\in [{n}]}{N_i \choose 2}$ achieves the minimum value when all $N_i$ have almost the same values, that is, $N_i$ equals $\lambda$ or $\lambda-1$, $i\in [{n}]$ by Eq.~(\ref{nonzeros}). Assume that there are $x$ columns such that $N_i=\lambda-1$. Then by Eq.~(\ref{nonzeros}), \[Mw=\lambda (Mw_1-x)+(\lambda-1)x.\] Hence $x=\lambda Mw_1-Mw=Ms$. By Eq.~(\ref{pairs}), we have \[{\lambda \choose 2}(Mw_1-Ms)+{\lambda-1\choose 2}Ms\leq {M\choose 2},\] which yields that $M\geq \lambda(\lambda-1)w_1-2(\lambda-1)s+1$. Let $\mu:=\lambda(\lambda-1)w_1-2(\lambda-1)s$. Then \[n\geq (\mu+1) w_1,\] that is,  $(\mu+1) w_1$ is the smallest possible length $n$ which is a multiple of $w_1$ such that $A_q(n,2w-1,\vw)=
\left\lfloor\dfrac{n}{w_1}\right\rfloor$.

Next, we deal with length $n$ such that $\mu w_1<n<(\mu+1) w_1$. Suppose that $|\C|=\mu$ and $n=\mu w_1+r$, where $1\leq r < w_1$. We need to find the smallest integer $r$ such that Eqs.~(\ref{nonzeros}) and (\ref{pairs}) both hold.  By doing the same arguments as the case when $w_1|n$, we deduce that $\lambda(\lambda-1)r\geq \mu$, that is, \[r\geq \lceil \frac{\mu}{\lambda(\lambda-1)}\rceil= w_1-\left\lfloor\frac{2s}{\lambda}\right\rfloor.\]
Since $r<w_1$, we need $2s\geq \lambda$ in this case.

Now we have proved the following lower bound on $N_{ccc}(\vw)$.
\begin{proposition}
\label{boundccc}  Given a composition $\vw$ with at least two components. Let $\lambda=\lceil\frac{w}{w_1}\rceil$ and $s=\lambda w_1-w$. Then \[N_{ccc}(\vw)\geq \mu w_1+ \lceil \frac{\mu}{\lambda(\lambda-1)}\rceil= (\mu+1) w_1-\left\lfloor\frac{2s}{\lambda}\right\rfloor,\]
where $\mu=\lambda(\lambda-1)w_1-2(\lambda-1)s$.
\end{proposition}

As mentioned above, $(\mu+1) w_1$ is the smallest possible integer $n$ which is a multiple of $w_1$ such that  $A_q(n,2w-1,\vw)=\lfloor\frac{n}{w_1}\rfloor$. Further,
\begin{align*}&(\mu+1) w_1-\left\lfloor\frac{2s}{\lambda}\right\rfloor= w^2-w_1(w-1)+(w_1s-s^2-\lfloor\frac{2s}{\lambda}\rfloor).
\end{align*}
Since $w_1s-s^2-\lfloor\frac{2s}{\lambda}\rfloor\geq 0$ when $\lambda\geq 2$, the lower bound in Proposition~\ref{boundccc} is stronger than that in Proposition~\ref{bound}.

Observe that in Proposition~\ref{boundccc}, the lower bound only depends on the total weight $w$ and the biggest component $w_1$. By Lemma~\ref{refine}, it is easy to prove the following fact.
\begin{lemma}\label{refine1} Suppose that $\vw$ is a refinement of $\vv$ such that $w_1=v_1$. If $N_{ccc}(\vv)$ achieves the lower bound in Proposition~\ref{boundccc},  so does $N_{ccc}(\vw)$.
\end{lemma}

Now we show that for ternary constant-composition codes, the lower bound in Proposition~\ref{boundccc} is always achievable.

\begin{proposition}
\label{nccc3}For all $w_1\geq w_2\geq 1$, we have  $N_{ccc}(\llbracket w_1, w_2\rrbracket)= 2w_1w_2+w_2$.
\end{proposition}
\begin{proof}Let $\vw=\llbracket w_1, w_2\rrbracket$. Then $\lambda=2$ and $s=w_1-w_2$. By Proposition~\ref{boundccc},  $N_{ccc}(\vw)\geq 2w_1w_2+w_2$. By \cite{Svanstrometal:2002},
\begin{align*}
A_3(n,2w-1,\vw)=\max\{M:n\geq M(w_1+\max\{w_2-\frac{M-1}{2},0\})\}.
\end{align*}
So we only need to check that for all $n\geq 2w_1w_2+w_2$, $A_3(n,2w-1,\vw)=\lfloor\dfrac{n}{w_1}\rfloor$. Let \[F(n):=n-M(w_1+\max\{w_2-\frac{M-1}{2},0\}),\] where $M=\lfloor\dfrac{n}{w_1}\rfloor$ is a function of $n$. Since $A_3(n,2w-1,\vw)\leq \lfloor\dfrac{n}{w_1}\rfloor$ for all $\vw$, it suffices to check that $F(n)\geq 0$ for all $n\geq 2w_1w_2+w_2$. We prove it by induction on $n$. It is easy to show that $F(2w_1w_2+w_2)=0$. Suppose that $F(n)\geq 0$ for some $n\geq 2w_1w_2+w_2$, we want to show that $F(n+1)\geq 0$ too. Let $n=Mw_1+r$, where $0\leq r <w_1$ and $M\geq 2w_2$. If $r<w_1-1$, then $n+1=Mw_1+r+1$. Hence $F(n+1)=F(n)+1\geq 1$. If $r=w_1-1$, then $n+1=(M+1)w_1$. Hence $F(n+1)= 0$. This completes the proof.
\end{proof}

\section{A Combinatorial Construction}\label{combcons}

In this section, we provide a general combinatorial construction for optimal $(n, 2w-1, \vw)_q$-codes of size $\lfloor\dfrac{n}{w_1}\rfloor$, when the length $n\geq (\mu+1) w_1-\left\lfloor\frac{2s}{\lambda}\right\rfloor$, where $\mu=\lambda(\lambda-1)w_1-2(\lambda-1)s$, $\lambda=\lceil\frac{w}{w_1}\rceil$ and $s=\lambda w_1-w$. From now on, we assume that $n\geq (\mu+1) w_1-\left\lfloor\frac{2s}{\lambda}\right\rfloor$. By Lemma~\ref{refine1}, we can also assume that the composition $\vw$  is not a refinement of any $\vv$ such that $w_1=v_1$.

 Note that $A_q(n,2w-1,\vw)\leq M$ for all length $n\in [Mw_1, Mw_1+w_1-1]$. If $2s<\lambda$, we only need to construct optimal codes for length $n$ which is a multiple of $w_1$, that is $n\in \{Mw_1: M\geq \mu+1\}$. If $2s\geq \lambda$, we also need to construct optimal codes for length $n= (\mu+1) w_1-\left\lfloor\frac{2s}{\lambda}\right\rfloor=\mu w_1+\lceil \frac{\mu}{\lambda(\lambda-1)}\rceil$. In this case, $\mu w_1<n<(\mu+1) w_1$, and the optimal codes have size upper bounded by $\mu$. For other length $n$, apply the lengthening method (adding zeros in the end of codewords) used in \cite{Cheeetal:2010a}. For convenience, let
 \[S(\vw)=\{(M,Mw_1): M\geq \mu+1\}\cup \{(\mu, \mu w_1+\lceil \frac{\mu}{\lambda(\lambda-1)}\rceil): \text{ if } 2s\geq \lambda\},\] which  is the collection of pairs $(M,n)$ that we need to construct an $(n, 2w-1, \vw)_q$-code of size $M$.

 Before giving our construction,  we introduce necessary terminology in combinatorial design theory.
A {\em set system} is a pair $(X,\B)$ such that $X$ is a finite set
of {\em points} and $\B$ is a set of subsets of $X$, called {\em
blocks}. The {\em order} of the set system is $|X|$, the number of
points. For a set of  nonnegative integers $K$, a set system $(X,\B)$
is said to be $K$-{\em uniform} if $|B|\in K$ for all $B\in\B$.

 A $(v,K)$-{\em packing} is a $K$-uniform set system
$(X,\B)$ of order $v$, such that each pair of $X$ occurs in at most
one block in $\B$. The {\em packing number} $D(v,K)$ is the maximum
number of blocks in any $(v,K)$-packing. A $(v,K)$-packing
$(X,\B)$ is said to be {\em optimal} if $|\B|=D(v,K)$. If $K=\{k\}$, then we write $k$ instead of $\{k\}$ for short. The values
of $D(v,k)$ have been determined for all $v$ when $k\in
\{3,4\}$ \cite{SWY:2007}. In particular, we have
\begin{equation}\label{eq1} D(v,3)=\begin{dcases}
\left\lfloor\frac{v}{3}\left\lfloor\frac{v-1}{2}\right\rfloor\right\rfloor-1, &\text{if $v\equiv 5 \pmod 6$;} \\
\left\lfloor\frac{v}{3}\left\lfloor\frac{v-1}{2}\right\rfloor\right\rfloor, &\text{otherwise. }
\end{dcases}\end{equation}
In fact, when $v\equiv 1,3 \pmod 6$, an optimal $(v,3)$-packing is also called {\em a Steiner triple system of order $v$}, denoted by STS$(v)$. In this case,
 each  pair of points occurs exactly once.

Suppose that $\C$ is an $(n,2w-1,\vw)_q$-code of size $M$, where $(M,n)\in S(\vw)$.
 As in Section~\ref{lbound},   $\C$  can be regarded as an $M\times n$ matrix $\bC$, whose  rows are codewords of $\C$.
For each column $c\in [{n}]$, we assume that $N_c=\lambda$ or $\lambda-1$, although it is not necessarily the case. Further, each entry from $[{q-1}]$  occurs at most once  in each column. Let the rows of $C$ be indexed by $\bz_{M}$.
Then we can define a $(q-1)$-tuple $A_c=(a_1,a_2,\ldots, a_{q-1})\in (\bz_{M}\cup\{*\})^{q-1}$ for each column $c$, with $a_i$ being the index of the row containing symbol $i$ in column $c$, $i\in [{q-1}]$. If some symbol $i$ does not occur in column $c$, then let $a_i=*$.  Let $\A=\{A_c:c\in [{n}]\}$, then
$\A$ satisfies the following properties.
\begin{description}
\item[(T1)] For each $c$, all the elements in $A_c$ excluding $*$  are  distinct. Let $B_c$ be the set containing all elements in $A_c$ excluding $*$ and $\B=\{B_c:c\in [{n}]\}$. Then
 $(\bz_{M},\B)$ is an $(M,\{\lambda,\lambda-1\})$-packing of size $n$.
\item[(T2)] For each position $i\in [{q-1}]$ (referring positions in $A_c$),  each element of $\bz_{M}$  occurs in  position $i$ exactly $w_i$  times in $\A$.
\end{description}

  \begin{example}\label{code2packing}Let $\vw=\llbracket 2,2,1\rrbracket$. Then $\lambda=3$ and $s=1$. The following  is the matrix form of an optimal $(18,9,\vw)_4$-code of size $9$ from \cite{Cheeetal:2010a}. Let the rows of $C$ be indexed with elements in $\bz_9$. Then the corresponding triples of $\A$ are listed below each column of $C$.

\begin{equation*}
C=\left[ \begin{array}{@{}*{18}{c}@{}}
1&0&0&1&2&0&0&0&0&2&0&3&0&0&0&0&0&0\\
0&0& 1&0&0&1&2&0&0&0&0&2&0&3&0&0&0&0\\
0&0&0&0& 1&0&0&1&2&0&0&0&0&2&0&3&0&0\\
0&0&0&0&0&0& 1&0&0&1&2&0&0&0&0&2&0&3\\
0&3&0&0&0&0&0&0& 1&0&0&1&2&0&0&0&0&2\\
0&2&0&3&0&0&0&0&0&0& 1&0&0&1&2&0&0&0\\
0&0&0&2&0&3&0&0&0&0&0&0& 1&0&0&1&2&0\\
2&0&0&0&0&2&0&3&0&0&0&0&0&0& 1&0&0&1\\
0&1&2&0&0&0&0&2&0&3&0&0&0&0&0&0& 1&0\\
\end{array} \right]
\end{equation*}

\begin{equation*}
\A:~ \begin{array}{@{}*{18}{c}@{}}
\downarrow&\downarrow&\downarrow&\downarrow&\downarrow&\downarrow& \downarrow&\downarrow&\downarrow&\downarrow&\downarrow&\downarrow&\downarrow&\downarrow&\downarrow&\downarrow&\downarrow&\downarrow\\
0&8&1&0&2&1&3&2&4&3&5&4&6&5&7&6&8&7\\
7&5&8&6&0&7&1&8&2&0&3&1&4&2&5&3&6&4\\
*&4&*&5&*&6&*&7&*&8&*&0&*&1&*&2&*&3\\
\end{array}
\end{equation*}
It is easy to check that $\A$ satisfies the properties (T1) and (T2).
\end{example}

The converse is true. Given a pair $(\bz_M,\A)$, where $\A\subset (\bz_{M}\cup\{*\})^{q-1}$. If  $\A$ satisfies (T1) and (T2) for a composition $\vw$, then
we can construct an $M\times |\A|$ matrix $\bC$ in a natural way, where the rows of $\bC$ form
  an $(|\A|,2w-1,\vw)_q$-code of size $M$. In fact, (T1) guarantees the code has minimum distance $2w-1$, while (T2)  guarantees each codeword is of constant composition $\vw$. Such a pair $(\bz_M,\A)$ is called a {\em $\vw$-balanced $(M,q-1)$-packing}.

  \begin{proposition}\label{pac2code}
  If there exists a $\vw$-balanced $(M,q-1)$-packing of size $n$, then there exists an $(n,2w-1,\vw)_q$-code of size $M$.
  \end{proposition}

We aim to construct  optimal $(n,2w-1,\vw)_q$-codes of size $M$ by establishing the existence of  $\vw$-balanced $(M,q-1)$-packings of size $n$ for $(M,n)\in S(\vw)$. By the similar arguments as in Section~\ref{lbound}, we can compute the numbers of blocks of sizes $\lambda$ and $\lambda-1$ in the $(M,\{\lambda,\lambda-1\})$-packing defined in  (T1). The details of these numbers are listed in the following table.
\begin{table}[h!]
\begin{center}  \caption{Distribution of block sizes $\lambda$ and $\lambda-1$}
\begin{tabular}{ |c| c |c| c|}
  \hline
  &$n$&\# blocks of size $\lambda$&\# blocks of size $\lambda-1$\\ \hline
  $M\geq \mu+1$&$Mw_1$&$Mw_1-Ms$&$Ms$\\ \hline
  $M=\mu$  &$M w_1+\lceil \frac{M}{\lambda(\lambda-1)}\rceil$ &$M w_1+(1-\lambda)\lceil \frac{M}{\lambda(\lambda-1)}\rceil-Ms$&$\lambda\lceil \frac{M}{\lambda(\lambda-1)}\rceil+Ms$\\ \hline
\end{tabular}

\label{table1}
\end{center}
\end{table}

The next two sections will study linear size quaternary constant-composition codes. Given a composition $\vw=\llbracket w_1, w_2,w_3\rrbracket$, if $w_1\geq w_2+w_3$, then $\vw$ is a refinement of $\llbracket w_1, w_2+w_3\rrbracket$. By Lemma~\ref{refine1} and Proposition~\ref{nccc3}, the value of $N_{ccc}(\vw)$ can be determined for this case. Hence  we assume that $w_1< w_2+w_3$, that is, $\mu=6w_1-4s$, $\lambda=3$ and $s=2w_1-w_2-w_3$ in the remaining of this paper. Distribution of different block sizes in this case is listed below.

\begin{table}[h!]
\begin{center}  \caption{Distribution of block sizes $3$ and $2$}
\begin{tabular}{ |c| c |c| c|}
  \hline
  &$n$&\# blocks of size $3$&\# blocks of size $2$\\ \hline
  $M\geq \mu+1$&$Mw_1$&$M(w_1-s)$&$Ms$\\ \hline
  $M=\mu$  &$M w_1+\lceil \frac{M}{6}\rceil$ &$M w_1-2\lceil \frac{M}{6}\rceil-Ms$&$3\lceil \frac{M}{6}\rceil+Ms$\\ \hline
\end{tabular}

\label{table2}
\end{center}
\end{table}

\section{Constructions for $n\geq (\mu+1)w_1$}\label{ngeql}

In this section, we show that $A_4(n,2w-1,\vw)=
\lfloor\dfrac{n}{w_1}\rfloor$ for all $n\geq (\mu+1)w_1=6w_1^2-4sw_1+w_1$ based on the existence of difference families.

\subsection{Difference Families}
Let $B=\{b_1,\ldots,b_k\}$ be a $k$-subset of $\bz_n$. The {\em list of differences from $B$} is the multiset $\Delta B=\langle b_i-b_j: i,j \in [k], i\neq j\rangle$.  A collection $\{B_1 , \ldots, B_t\}$ of $k$-subsets of  $\bz_n$ forms an
$(n, k; t)$ {\em difference packing}, or $t$-DP$(n,k)$, if every nonzero element of $\bz_n$
occurs at most once  in $\Delta B_1 \cup\cdots \cup \Delta B_t$. The sets $B_i$ are {\em base blocks}.  If every nonzero element of $\bz_n$
occurs exactly once in $\Delta B_1 \cup\cdots \cup \Delta B_t$, it is  known as an $(n, k)$ {\em difference family}, or DF$(n,k)$ \cite{abel2007difference}. The parameter $t$ is omitted since it could be computed from $n$ and $k$, that is, $t=\frac{n-1}{k(k-1)}$. Since $t$ must be an integer, for a DF$(n,k)$ exists,  we must have $n\equiv 1 \pmod{k(k-1)}$.

The sizes of base blocks are same in a difference packing. It is natural to  generalize difference packings to a collection of subsets with the same property but with {\em varying} block sizes. If $t=e_1+\ldots+e_s$, and if there are $e_i$ base blocks of size $k_i$, then the  generalized difference packing is of {\em block type} $k_1^{e_1}\cdots k_s^{e_s}$, and denoted by GDP$(n,k_1^{e_1}\cdots k_s^{e_s})$. Without loss of generality, we assume that $k_1\geq\cdots\geq k_s\geq 2$.
%
%
%
%
%
%

Given a triple $A=(a_1,a_2,a_3)\in (\bz_M\cup\{*\})^3$, define \[{\rm Orb}_{\bz_M}A=\{(a_1+i,a_2+i,a_3+i): i\in \bz_M\},\] where $*+i=*$ for any $i\in \bz_M$.

\begin{proposition}
\label{diffpack}  Suppose that there exists a  GDP$(M,3^{e_1}2^{e_2})$. Let $w_1=e_1+e_2$, $w_2$ and $w_3$ be  integers such that $w_1\geq w_2\geq w_3$ and $w_2+w_3=2e_1+e_2$. Then there exists a $\vw$-balanced $(M,3)$-packing of size $n$, where $n=w_1M$ and $\vw=\llbracket w_1, w_2,w_3 \rrbracket$.
\end{proposition}
\begin{proof} Given a GDP$(M,3^{e_1}2^{e_2})$, partition the set $\B$ of base blocks into three parts $\B_1$, $\B_2$ and $\B_3$, where $\B_1$ consists of all $e_1$ blocks of size three,  $\B_2$ contains $w_2-e_1$ blocks of size two, and  $\B_3$ contains the remaining $w_3-e_1$ blocks of size two. For any $B=\{a,b,c\}\in \B_1$, define $A_{B}=(a,b,c)$; for any $B=\{a,b\}\in \B_2$, define $A_{B}=(a,b,*)$ and; for any $B=\{a,c\}\in \B_3$, define $A_{B}=(a,*,c)$. Let $\A=\cup_{B\in \B} {\rm Orb}_{\bz_M}A_B$, then $(\bz_M,\A)$ is   a $\vw$-balanced $(M,3)$-packing of size $n$, where $n=w_1M$ and $\vw=\llbracket w_1, w_2,w_3 \rrbracket$.
\end{proof}

In a DF$(n,k)$,  the $t$ blocks $B_i = \{b_{i,1} , \ldots , b_{i,k}\}$, $i\in [t]$, form a {\em perfect} $(n, k)$
difference family over $\bz_n$ if the $tk(k - 1)/2$ differences $b_{i,h}-b_{ i,g}$  ($i \in [t]$,
$1 \leq  g < h \leq k$) cover the set $\{1, 2, \ldots, (n- 1)/2\}$. If instead, they cover the set
$\{1, 2, \ldots, (n-3)/2\} \cup \{(n + 1)/2\}$, then the difference family is {\em quasi-perfect}. We denote them by PDF$(n,k)$ and quasi-PDF$(n,k)$ respectively. The existences of PDF$(n,k)$s and quasi-PDF$(n,k)$s are known when $k=3$.

\begin{theorem}\label{pdf}\cite{abel2007difference}
A PDF$(n,3)$
exists when $n \equiv 1 \text{ or } 7\pmod {24}$, and a quasi-PDF$(n,3)$
exists when $n \equiv 13 \text{ or } 19\pmod {24}$.
\end{theorem}

\begin{corollary}\label{gpdf}
Let $e_1, e_2\geq 0$ be two integers. Then a GDP$(M,3^{e_1}2^{e_2})$ exists for all  $M\geq 6e_1+2e_2+1$ except when  $e_1 \equiv 2\text{ or } 3\pmod {4}$ and  $(M,e_2)=(6e_1+2,0)$.
\end{corollary}
\begin{proof}For each  $e_1 \equiv 0\text{ or } 1\pmod {4}$, let $m=6e_1+1$. By Theorem~\ref{pdf}, there exists a PDF$(m,3)$ over $\bz_m$.
Let $\B$ be the collection of all $e_1$ base blocks.
Given any $e_2\geq 0$,  let $P_i=\{0, \frac{m-1}{2}+i\}$, $i\in [{e_2}]$. Then $\B\cup \{P_i: i\in [{e_2}]\}$ is a GDP$(M,3^{e_1}2^{e_2})$
 for all $M\geq m+2e_2$.

 For each  $e_1 \equiv 2\text{ or } 3\pmod {4}$, let $m=6e_1+1$. By Theorem~\ref{pdf}, there exists a quasi-PDF$(m,3)$ over $\bz_m$, which is also
   a GDP$(M,3^{e_1})$ for all $M\geq m$ except when $M= m+1$.
Let $\B$ be the collection of all $e_1$ base blocks.
Given any $e_2\geq 1$,  let $P_1=\{0, \frac{m-1}{2}\}$ and $P_i=\{0, \frac{m-1}{2}+i\}$ for all $i\in [{2,e_2}]$. Then $\B\cup \{P_i: i\in [{e_2}]\}$ is a GDP$(M,3^{e_1}2^{e_2})$
 for all $M\geq m+2e_2$.
\end{proof}
By the relations among all parameters in Proposition~\ref{diffpack}, it is easy to show that $6e_1+2e_2+1=6w_1-4s+1=\mu+1$. Combining Corollary~\ref{gpdf}, Propositions~\ref{pac2code} and~\ref{diffpack}, it is immediate that the following result holds.
\begin{proposition}\label{gpdf2}
Let $\vw=\llbracket w_1, w_2,w_3 \rrbracket$.  Then $A_4(n,2w-1,\vw)=\lfloor\dfrac{n}{w_1}\rfloor$  for all $n= Mw_1$, where $M\geq \mu+1$, except when $w_1=w_2=w_3\equiv 2\text{ or } 3\pmod {4}$ and $n=6w_1^2+2w_1$.
\end{proposition}

\subsection{Exceptions in Proposition~\ref{gpdf2}}\label{subexcept}
Now we settle the exceptional cases in Proposition~\ref{gpdf2}. That is, we need to prove that \[A_4(6w_1^2+2w_1,2w-1,\llbracket w_1, w_1,w_1 \rrbracket)=6w_1+2\] for all $w_1\equiv 2\text{ or } 3\pmod {4}$.

By  Proposition~\ref{pac2code}, we need to construct a $\vw$-balanced $(6w_1+2,3)$-packing of size $6w_1^2+2w_1$ for all $w_1\equiv 2\text{ or } 3\pmod {4}$, where $\vw=\llbracket w_1, w_1,w_1 \rrbracket$. Actually, they exist for all positive integers $w_1$. Before stating our general construction, we give a small example first. Note that in this case, $s=0$, so there are only blocks of size three in the $(6w_1+2,\{3,2\})$-packing by Table~\ref{table2}, which is further optimal by the packing number in Eq.~(\ref{eq1}).

  \begin{example}\label{ext1t2}Let $w_1=2$ and $G=\bz_5\oplus \bz_3$. Write $x_y$ for the pair $(x,y)\in \bz_5\oplus \bz_3$. Let $B_0=(0_0,0_1,0_2)$, $B_1=(0_0,2_0,1_1)$ and $B_2=(0_0,4_0,2_1)$.  Let $\B'=\cup_{i\in [{0,2}]}${\rm Orb}$_G B_i$. If we consider all triples in $\B'$ as unordered $3$-subsets, then $\B'$ is the block set of an STS$(15)$ over $G$ due to Skolem \cite{Skolem:1927}.  Let $\B=\B'\setminus\{B\in \B': 4_2\in B\}$ and $X=G\setminus\{4_2\}$, then $(X,\B)$ is an optimal $(14,3)$-packing of size $28$ if again consider triples as unordered sets. We show the reordering procedures in Table~\ref{abn00}.

\begin{table}[h!] 
\center\caption{Reordering Procedures in Example~\ref{ext1t2}}\label{abn00}
\begin{tabular}{c|c|c|c}
\hline
 \multicolumn{2}{c|}{{\rm Orb}$_G B_0$}&  \multicolumn{2}{c}{reordering} \\
\hline
 \multicolumn{2}{c|}{($0_0, 0_1,0_2$)}&  \multicolumn{2}{c}{$\rightarrow(0_2, 0_1,0_0$)} \\
  \multicolumn{2}{c|}{($1_0, 1_1,1_2$)}&  \multicolumn{2}{c}{$\rightarrow(1_2, 1_1,1_0$)} \\
   \multicolumn{2}{c|}{($2_0, 2_1,2_2$)}&  \multicolumn{2}{c}{$\rightarrow(2_1,2_2, 2_0 $)} \\
    \multicolumn{2}{c|}{($3_0, 3_1,3_2$)}&  \multicolumn{2}{c}{$\rightarrow( 3_1,3_2, 3_0$)} \\
     \multicolumn{2}{c|}{\st{($4_0, 4_1,4_2$)}}&  \multicolumn{2}{c}{} \\
\hline
 {\rm Orb}$_G B_1$& reordering & {\rm Orb}$_G B_2$& reordering\\
\hline
($0_0,2_0,1_1$)&&($0_0,4_0,2_1$)&\\
($1_0,3_0,2_1$)&&($1_0,0_0,3_1$)&\\
($2_0,4_0,3_1$)&&($2_0,1_0,4_1$)&\\
($3_0,0_0,4_1$)&&($3_0,2_0,0_1$)&\\
($4_0,1_0,0_1$)&&($4_0,3_0,1_1$)&\\

($0_1,2_1,1_2$)&&($0_1,4_1,2_2$)&\\
($1_1,3_1,2_2$)&&($1_1,0_1,3_2$)&\\
($2_1,4_1,3_2$)&&\st{($2_1,1_1,4_2$)}&\\
\st{($3_1,0_1,4_2$)}&&($3_1,2_1,0_2$)&\\
($4_1,1_1,0_2$)&&($4_1,3_1,1_2$)&\\

($0_2,2_2,1_0$)&&\st{($0_2,4_2,2_0$)}&\\
($1_2,3_2,2_0$)&$\rightarrow(3_2,1_2,2_0$)&($1_2,0_2,3_0$)&\\
\st{($2_2,4_2,3_0$)}&&($2_2,1_2,4_0$)&\\
($3_2,0_2,4_0$)&&($3_2,2_2,0_0$)&$\rightarrow(2_2,3_2,0_0$)\\
\st{($4_2,1_2,0_0$)}&&\st{($4_2,3_2,1_0$)}&\\
\hline
\end{tabular}
\end{table}
 First look at the set {\rm Orb}$_G B_1\cup${\rm Orb}$_G B_2$, in which each element from $G$ occurs twice in each position. After deleting triples containing the element $4_2$,  elements $2_1,3_1,0_2,2_2$ occurs only once in the first position, elements $0_1,1_1,1_2,3_2$ occurs only once in the second position, and elements $0_0,1_0,2_0,3_0$ occurs only once in the third position. It is natural to think of reordering triples from {\rm Orb}$_G B_0$ to increase the occurrences of these elements. After reordering the remaining triples of {\rm Orb}$_G B_0$ as in Table~\ref{abn00},  element $1_2$ occurs three times in the first position but only once in the second position, while element $2_2$ occurs only once in the first position but three times in the second position. Finally, we exchange the first two elements in the triple ($1_2,3_2,2_0$) from {\rm Orb}$_G B_1$ and ($3_2,2_2,0_0$) from {\rm Orb}$_G B_2$ to balance the occurrences.

\end{example}
\begin{proposition}\label{packing}
For all positive integers $w_1$, there exists a $\vw$-balanced $(6w_1+2,3)$-packing of size $6w_1^2+2w_1$, where  $\vw=\llbracket w_1, w_1,w_1 \rrbracket$.
\end{proposition}
\begin{proof}
Let $u=2w_1+1$. We start from an STS$(3u)$ which is due to Skolem \cite{Skolem:1927}.
Let $G=\bz_u\oplus \bz_3$.
Choose base blocks $A_0=(0_0,0_1,0_2)$ and $A_x=(0_0,(2x)_0,x_1)$, $x\in [{w_1}]$. It is easy to see that $|${\rm Orb}$_G A_0|=u$ and $|${\rm Orb}$_G A_x|=3u$, $x\in [{w_1}]$. Let  $\A'=\cup_{x\in [{0,w_1}]} ${\rm Orb}$_G A_x$, $\A=\A'\setminus\{A\in \A': (2w_1)_2\in A\}$ and $X=G\setminus \{(2w_1)_2\}$. If we consider triples as unordered sets, 
  then $(G,\A')$ is an STS$(3u)$, and
 $(X,\A)$ is an optimal $(6w_1+2,3)$-packing of size $6w_1^2+2w_1$.

 Note that $\A$ doesn't satisfy (T2) at this
moment.  The first three rows of Table~\ref{abn0} point out the sets of elements occurring $w_1-1$ or $w_1+1$ in each position, all others occur in the corresponding positions exactly $w_1$ times in $\A$.
The  $w_1-1$ occurrences happen when deleting triples from {\rm Orb}$_G A_x$, $x\in [{w_1}]$, while  $w_1+1$ happens because of triples from {\rm Orb}$_G A_0$.

For convenience, denote $\iota=\lfloor\frac{w_1}{2}\rfloor$ and $\kappa=\lceil\frac{w_1}{2}\rceil$. We  follow the steps below.

\begin{description}
\item[(S1)] In {\rm Orb}$_G A_0$, change $((2i+1)_0,(2i+1)_1,(2i+1)_2)$ to
 $((2i+1)_1,(2i+1)_2,(2i+1)_0)$ for all $i\in [{\iota,w_1-1}]$.
\item[(S2)] In {\rm Orb}$_G A_0$, change $((2i)_0,(2i)_1,(2i)_2)$ to
 $((2i)_1,$ $(2i)_2,(2i)_0)$ for all $i\in [{\kappa,w_1-1}]$.
 \item[(S3)] In {\rm Orb}$_G A_0$, change  $(i_0,i_1,i_2)$ to
$(i_2,i_1,i_0)$ for all $i\in [{0,w_1-1}]$.
 \item[(S4)] Finally, in {\rm Orb}$_G A_{\kappa}$, change  $((2i+1)_2,(2\kappa+2i+1)_2,(\kappa+2i+1)_0)$ to
 $((2\kappa+2i+1)_2,(2i+1)_2,(\kappa+2i+1)_0)$ for all $i\in [{0,\iota-1}]$.
 At the same time, in {\rm Orb}$_G A_{w_1}$, change  $((2i+1)_2,(2i)_2,(w_1+2i+1)_0)$ to $((2i)_2,(2i+1)_2,(w_1+2i+1)_0)$ for all
 $i\in [{\kappa,w_1-1}]$. Note that $\{(2\kappa+2i+1)_2:i\in [{0,\iota-1}]\}=\{(2i+1)_2:i\in [{\kappa,w_1-1}]\}$.
\end{description}

After each step, all elements occur at least $w_1-1$ and at most $w_1+1$ times in each position. We list the elements occurring $w_1-1$ or $w_1+1$ in each position after each step in Table~\ref{abn0}.  It is routine to check that after (S4), all elements occur $w_1$ times in each position. Thus triples in $\A$ can be reordered to satisfies (T2).

\begin{table*}[h!] 
\caption{Abnormal occurrences in three positions in Proposition~\ref{packing}}\label{abn0}
\makebox[\textwidth][c]{ 
\begin{tabular*}{0.86\paperwidth}{ @{\extracolsep{\fill}}  c | c | c | c }
\hline
\text{Originally} & $1$& $2$&$3$\\
\hline
$w_1-1$&$\{(2i)_2: i\in [{0,w_1-1}]\}\cup \{(w_1)_1:i\in [{w_1,2w_1-1}]\}$  & $\{(2i+1)_2:i\in [{0,w_1-1}]\}$&$\{i_0:i\in [{0,2w_1-1}]\}$\\
\hline
$w_1+1$& $\{i_0:i\in [{0,2w_1-1}]\}$ &  $\{(w_1)_1:i\in [{w_1,2w_1-1}]\}$& $\{i_2:i\in [{0,2w_1-1}]\}$\\
\hline
\hline
\text{After (S1)}& $1$& $2$&$3$\\
\hline
$w_1-1$&$\{(2i)_2: i\in [{0,w_1-1}]\}\cup\{(2i)_1:i\in [{\kappa,w_1-1}]\}$  & $\{(2i+1)_2:i\in [{0,\iota-1}]\}$&$\{i_0:i\in [{0,w_1-1}]\}\cup \{(2i)_0:i\in [{\kappa,w_1-1}]\}$\\
\hline
$w_1+1$& $\{i_0:i\in [{0,w_1-1}]\}\cup \{(2i)_0:i\in [{\kappa,w_1-1}]\}$ &  $\{(2i)_1:i\in [{\kappa,w_1-1}]\}$& $\{i_2:i\in [{0,w_1-1}]\}\cup \{(2i)_2:i\in [{\kappa,w_1-1}]\}$\\
\hline
\hline
\text{After (S2)}& $1$& $2$&$3$\\
\hline
$w_1-1$&$\{(2i)_2: i\in [{0,w_1-1}]\}$  & $\{(2i+1)_2:i\in [{0,\iota-1}]\}$&$\{i_0:i\in [{0,w_1-1}]\}$\\
\hline
$w_1+1$& $\{i_0:i\in [{0,w_1-1}]\}$ &  $\{(2i)_2:i\in [{\kappa,w_1-1}]\}$& $\{i_2:i\in [{0,w_1-1}]\}$\\
\hline
\hline
\text{After (S3)}& $1$& $2$&$3$\\
\hline
$w_1-1$&$\{(2i)_2: i\in [{\kappa,w_1-1}]\}$  & $\{(2i+1)_2:i\in [{0,\iota-1}]\}$&$\emptyset$\\
\hline
$w_1+1$& $\{(2i+1)_2:i\in [{0,\iota-1}]\}$ &  $\{(2i)_2:i\in [{\kappa,w_1-1}]\}$& $\emptyset$\\
\hline
\end{tabular*}
}
\end{table*}
 \end{proof}
The following consequence is immediate.
 \begin{proposition}\label{gpdf3} For all positive integers $w_1$ and $n=6w_1^2+2w_1$,
 \[A_4(n,2w-1,\llbracket w_1, w_1,w_1 \rrbracket)=6w_1+2.\]
\end{proposition}

Combining Propositions~\ref{gpdf2} and \ref{gpdf3}, we have shown that $A_4(n,2w-1,\llbracket w_1, w_2,w_3 \rrbracket)=\lfloor\dfrac{n}{w_1}\rfloor$ for all $n= Mw_1$, where $M\geq \mu+1$. By the lengthening method (adding zeros in the end of codewords) used in \cite{Cheeetal:2010a}, we have $A_4(n,2w-1,\llbracket w_1, w_2,w_3 \rrbracket)=\lfloor\dfrac{n}{w_1}\rfloor$  for all $n\geq (\mu+1)w_1$.

\section{Constructions for $n=\mu w_1+\lceil \frac{\mu}{6}\rceil$}\label{nleql}
To determine values of $N_{ccc}(\vw)$, we still need to prove that $A_4(n,2w-1,\llbracket w_1, w_2,w_3 \rrbracket)=\lfloor\dfrac{n}{w_1}\rfloor$  for $n=\mu w_1+\lceil \frac{\mu}{6}\rceil$ if $2s\geq \lambda=3$.  From now on, we assume that $s\geq 2$, that is $2w_1\geq w_2+w_3+2$.

By  Proposition~\ref{pac2code}, we need to construct a $\vw$-balanced $(\mu,3)$-packing of size $\mu w_1+\lceil \frac{\mu}{6}\rceil$ for all $\vw=\llbracket w_1, w_2,w_3 \rrbracket$ satisfying that  $s\geq 2$.
Here we use different method from that in Proposition~\ref{packing}. We first find a candidate of $\A$ satisfying (T2), then try to modify it to satisfy (T1). We show this idea in the following example.

\begin{example}Let $M=10$ and $\vw=\llbracket 3, 2,2 \rrbracket$.
Let $A_1=(0,1,6)$, $A_2=(0,2,*)$ and $A_3=(0,*,3)$. Then $\A=\cup_{i\in [3]} ${\rm Orb}$_{\bz_{10}}A_i$ is a candidate satisfying (T2) over $\bz_{10}$. Note that the difference $5$ occurs twice in $\Delta A_1$. We first do the following changes to triples  in {\rm Orb}$_{\bz_{10}}A_1$:
\begin{align*}
(0,1,6) &\rightarrow(0,*,6),\\
(1,2,7) &\rightarrow(1,2,*),\\
(3,4,9) &\rightarrow(3,*,9),\\
(4,5,0)&\rightarrow(4,5,*),\\
(7,8,3)&\rightarrow(7,8,*).
\end{align*}
 Then add two more triples
 \begin{align*}
 &(*,1,0),\\&(*,4,3).
 \end{align*}
Finally, change the following triple in {\rm Orb}$_{\bz_{10}}A_2$
  \begin{align*}
(1,3,*)&\rightarrow(1,3,7).
 \end{align*}
Note that we do not change the positions of symbols from $\bz_{10}$ appearing in $\A$. For example, in the first triple $(0,1,6)$, the symbol $1$ in the second position disappears, but appears later in $(*,1,0)$ in the same position. Further, the pairs newly occurring  in the last two steps are pairs deleted in the first step. For example, the pair $\{0,1\}$ appears in the second step when adding $(*,1,0)$, but it was deleted before in the first step when changing $(0,1,6)$ to $(0,*,6)$. So the pair $\{0,1\}$ still occurs only once after these three steps. Thus it is easy to check that we have a set of $32$ triples satisfying both (T1) and (T2), which yields that
$A_4(32,13,\llbracket 3, 2,2 \rrbracket)=10$ and $N_{ccc}(\llbracket 3, 2,2 \rrbracket)=32$.
\end{example}

\begin{proposition}
\label{gpack}  Let $e_1\geq 0$, $e_2\geq 3$, $2|M$ and $M\geq 8$ be integers. Suppose that there exists a  GDP$(M,3^{e_1}2^{e_2})$ with three specified base blocks $\{0,1\}$, $\{0,2\}$ and $\{0,M/2-1\}$. Let $w_1=e_1+e_2-1$, $w_2$ and $w_3$ be  any integers such that $w_1\geq w_2\geq w_3$ and $w_2+w_3=2e_1+e_2$. Then there exists a $\vw$-balanced $(M,3)$-packing of size $n=M w_1+\lceil\frac{M}{6}\rceil$, where $\vw=\llbracket w_1, w_2,w_3 \rrbracket$.
\end{proposition}
\begin{proof} Suppose that $\B'$ is the given set of base blocks of a GDP$(M,3^{e_1}2^{e_2})$ over $\bz_M$. Let $B_1=\{0,1,M/2+1\}$, $B_2=\{0,2\}$ and \[\B=(\B'\setminus\{\{0,1\}, \{0,M/2-1\}\})\cup \{B_1\}.\]
Note that $\B$ has $e_1+e_2-1=w_1$ blocks.
Partition $\B$ into three parts $\B_1$, $\B_2$ and $\B_3$, where $\B_1$ consists of all $e_1+1$ blocks of size three,  $\B_2$ contains $w_2-e_1-1$ blocks of size two including  $B_2$ specifically, and  $\B_3$ contains the remaining $w_3-e_1-1$ blocks of size two. For each block $B\in \B_1$, let $A_B$ be an ordered triple with elements from $B$. For each block $B=\{a,b\}\in \B_2$, let $A_B=(a,b,*)$. For each block $B=\{a,c\}\in \B_3$, let $A_B=(a,*,c)$. Specifically,  let $A_{B_1}=(0,1,M/2+1)$ and $A_{B_2}=(0,2,*)$. Let $\A=\cup_{B\in \B}${\rm Orb}$_{\bz_{M}}A_B$, then $\A$  is a candidate of $M w_1$ triples over $\bz_{M}\cup\{*\}$ satisfying (T2).

Now we do modifications on triples in $\A$ to make it satisfy both (T1) and (T2).  The main idea is as follows. The set $\A$ does not satisfy (T1) since the difference $M/2$ occurs twice in $\Delta B_1$. Thus, we first choose $M/2$ triples from {\rm Orb}$_{\bz_{M}}A_{B_1}$, then change one symbol from each repeated pair to $*$. Besides the $M/2$ repeated pairs, there are $M/2$ other pairs also broken in this step. We let them appear somewhere else by adding $\lceil\frac{M}{6}\rceil$ triples of type $(*,\cdot , \cdot)$, and changing the symbol $*$ in $M/2-2\lceil\frac{M}{6}\rceil$ triples from {\rm Orb}$_{\bz_{M}}A_{B_2}$ to some symbol of $\bz_M$. Details for different congruent classes of $M$ are listed in Table~\ref{abn1}.

\begin{table*}[htbp] 
\center\caption{Modifications made on triples of $\A$ in Proposition~\ref{gpack}}\label{abn1}
\begin{tabular}{l|l|l}
\hline
\multirow{ 3}{*}{$M=6k$, $k\geq 2$}&{\rm Orb}$_{\bz_{M}}A_{B_1}$ & $(3i,3i+1,3i+3k+1)\rightarrow (3i,3i+1,*)$,  $i\in [{0,k-1}]$\\
&& $(3i+3k-1,3i+3k,3i)\rightarrow (3i+3k-1,3i+3k,*)$,  $i\in [{0,k-1}]$\\
&& $(3i+3k+1,3i+3k+2,3i+2)\rightarrow (3i+3k+1,*,3i+2)$,  $i\in [{0,k-1}]$\\
\cline{2-3}
&Add triples & $(*,3i+3k+2,3i+3k+1)$, $i\in [{0,k-1}]$\\
\cline{2-3}
&{\rm Orb}$_{\bz_{M}}A_{B_2}$ & $(3i+3k-1,3i+3k+1,*)\rightarrow (3i+3k-1,3i+3k+1,3i)$,  $i\in [{0,k-1}]$\\
\hline
\hline
\multirow{ 3}{*}{$M=6k+2$, $k\geq 1$}
&{\rm Orb}$_{\bz_{M}}A_{B_1}$ & $(3i+1,3i+2,3i+3k+3)\rightarrow (3i+1,3i+2,*)$,  $i\in [{0,k-1}]$\\
&& $(3i+3k,3i+3k+1,3i)\rightarrow (3i+3k,3i+3k+1,*)$,  $i\in [{0,k-1}]$\\
&& $(3i,3i+1,3i+3k+2)\rightarrow (3i,*,3i+3k+2)$,  $i\in [{0,k-1}]$\\
&& $(6k,6k+1,3k)\rightarrow (6k,*,3k)$\\
\cline{2-3}
&Add triples & $(*,3i+1,3i)$, $i\in [{0,k-1}]$\\
&& $(*,6k+1,6k)$\\
\cline{2-3}
&{\rm Orb}$_{\bz_{M}}A_{B_2}$ & $(3i+1,3i+3,*)\rightarrow (3i+1,3i+3,3i+3k+3)$,  $i\in [{0,k-2}]$\\
\hline
\hline
\multirow{ 3}{*}{$M=6k+4$, $k\geq 1$}
&{\rm Orb}$_{\bz_{M}}A_{B_1}$ & $(3i+1,3i+2,3i+3k+3)\rightarrow (3i+1,3i+2,*)$,  $i\in [{0,k-1}]$\\
&& $(3i+3k+1,3i+3k+2,3i)\rightarrow (3i+3k+1,3i+3k+2,*)$,  $i\in [{0,k}]$\\
&& $(3i,3i+1,3i+3k+3)\rightarrow (3i,*,3i+3k+3)$,  $i\in [{0,k}]$\\
\cline{2-3}
&Add triples & $(*,3i+1,3i)$, $i\in [{0,k-1}]$\\
\cline{2-3}
&{\rm Orb}$_{\bz_{M}}A_{B_2}$ & $(3i+1,3i+3,*)\rightarrow (3i+1,3i+3,3i+3k+4)$,  $i\in [{0,k-1}]$\\
\hline
\end{tabular}
\end{table*}
\end{proof}

By Proposition~\ref{gpack}, we need to construct a  GDP$(\mu,3^{e_1}2^{e_2})$ with three specified base blocks $\{0,1\}$, $\{0,2\}$ and $\{0,\mu/2-1\}$, where $e_1=w_1-s-1$ and $e_2=s+2$, for all $w_1>s\geq 2$.
\begin{proposition}\label{exgdp} Given integers  $w_1>s\geq 2$ and $(w_1,s)\not \in\{(4,2),(5,2)\}$, let  $e_1=w_1-s-1$ and $e_2=s+2$. Then there exists a  GDP$(\mu,3^{e_1}2^{e_2})$ over $\bz_{\mu}$ with three specified base blocks $\{0,1\}$, $\{0,2\}$ and $\{0,\mu/2-1\}$.
\end{proposition}
\begin{proof}
We split it into four cases based on the values of $w_1-s$.  For all cases, the $e_2$ base blocks of size two are of type $\{0,d\}$, where $d$ covers the values of all differences that do not appear in the list of differences from base blocks of size three. To save space, we only list the $e_1$ base blocks of size three in each case. Note that the differences $1$, $2$, $\mu/2$ and $\mu/2-1$ do not appear in any base block of size three. Thus the GDP contains the three mentioned base blocks of size two. Note that in the first two cases, that is when $w_1-s\equiv 0 \text{ or }1\pmod 4$, the set of base blocks of size three are obtained by modifying some blocks of a DF$(\mu,3)$ in \cite{dinitz1997disjoint}.

When $w_1-s=4k$, then $\mu=24k+2s$ and $e_1=4k-1$. If $k\geq 2$, then use $e_1$ base blocks of size three as below.
\[\begin{array}{llll}
\{0,&6k-1,&18k+2s-1\},&\\
 \{0,& 4k -1 ,& 9k- 1\},&\\
 \{0,&  2k ,& 10k - 1\},& \\
 \{0 ,& 4k ,& 10k\},& \\
 \{0 ,& 2k + 2r - 1 ,& 7k + r - 1\},&  r \in[1, k -1],\\
 \{0,&  2k + 2r,&  11k + r - 1\},&  r \in[1, k - 1],\\
 \{0 ,& 2r + 1 ,& 10k + r \},& r \in[1, k - 1],\\
 \{0 ,& 2r ,& 6k + r \},& r \in[2, k - 1].\\
\end{array}
\]
If $k=1$, and $s=2$, that is $\mu=28$, then three base blocks are $ \{0,3,11\}, \{0,4,9\},\{0,6,16\}$. If $s\geq 3$, then three base blocks are $\{0,3,13\},\{0,4,12\},\{0,5,11\}$.

When $w_1-s=4k+1$, then $\mu=24k+2s+6$ and $e_1=4k$. If $k\geq 2$, then the $e_1$ base blocks of size three are
\[\begin{array}{llll}
\{0,&6k,&18k+2s+4\},&\\
 \{0,& 4k - 1,& 9k\},&\\
 \{0 ,&4k,& 10k + 1\},&\\
 \{0,& 4k + 1,& 12k + 1\},&\\
 \{0,& 2k,& 12k\},&\\
 \{0,& 2k + 2r - 1,& 7k + r\},& r \in[1, k - 1],\\
 \{0,& 2k + 2r,& 11k + r\},& r \in[1, k - 1],\\
 \{0,& 2r + 1,& 10k + r + 1\},& r \in[1, k - 1],\\
 \{0,& 2r,& 6k + r + 1\},& r \in[2, k - 1].\\
\end{array}
\]
If $k=1$, and  $s\geq 2$, then four base blocks are $\{0,3,14\},\{0,4,12\},\{0,5,15\},\{0,6,13\}$.

When $w_1-s=4k+2$, then $\mu=24k+2s+12$ and $e_1=4k+1$. If $k\geq 1$, then the $e_1$ base blocks are
\[\begin{array}{llll}
\{0,&4k+1,&10k+6\},&\\
 \{0,& 4k +3,& 10k+7\},&\\
 \{0,&  2r ,& 6k + r+5\},& r \in[2, 2k+ 1],\\
 \{0,&  2r+1,& 10k + r+7\},& r \in[1, k],\\
 \{0,& 2k+2r + 1,& 11k + r + 7\},& r \in[1, k - 1].\\
\end{array}
\]
If $k=0$, and  $s\geq 3$, then let $\{0,3,7\}$ be the only base block of size three. 

When $w_1-s=4k+3$, then $\mu=24k+2s+18$ and $e_1=4k+2$. If $k\geq 1$, then the $e_1$ base blocks are
\[\begin{array}{llll}
\{0,&4k+1,&10k+9\},&\\
 \{0,& 4k +3,& 10k+10\},&\\
  \{0,& 4k +4,& 8k+9\},&\\
 \{0,&  2r ,& 6k + r+7\},& r \in[2, 2k+ 1],\\
 \{0,&  2r+1,& 10k + r+10\},& r \in[1, k-1], \\
 \{0,& 2k+2r - 1,& 11k + r + 9\},& r \in[1, k].\\
\end{array}
\]
If $k=0$, and  $s\geq 3$, then two base blocks are $\{0,3,8\},\{0,4,10\}$. 
\end{proof}

Combining Propositions~\ref{pac2code},~\ref{gpack} and \ref{exgdp}, we obtain the following result.
\begin{proposition}\label{nless}Given a composition $\vw=\llbracket w_1, w_2,w_3 \rrbracket$ such that $2\leq s<w_1$, where $s=3w_1-w$.  We have
\[A_4(\mu w_1+\lceil \frac{\mu}{6}\rceil,2w-1,\llbracket w_1, w_2,w_3 \rrbracket)=\mu,\] for all $w_1\geq w_2\geq w_3\geq 1$ such that  $(w_1,s)\not \in\{(4,2),(5,2)\}$.
\end{proposition}

By Propositions~\ref{gpdf2}, \ref{gpdf3}, \ref{nless} and the lengthening method, we have determined the value of $N_{ccc}(\vw)$ for almost all compositions with three components. We state it in the following proposition.
\begin{proposition}\label{nccc4}Given a composition $\vw=\llbracket w_1, w_2,w_3 \rrbracket$ such that $w_1< w_2+ w_3$. We have
\[N_{ccc}(\vw)=6w_1^2-4sw_1+w_1-\lfloor\frac{2s}{3}\rfloor,\]
where $s=3w_1-w$, except possibly when $\vw\in \{\llbracket 4,4,2 \rrbracket, \llbracket 4,3,3 \rrbracket, \llbracket 5,5,3 \rrbracket, \llbracket 5,4,4 \rrbracket\}$.
\end{proposition}

\section{Conclusion}

New direct constructions for optimal quaternary constant-composition codes have been given based on combinatorial methods. Consequently, we determine the values of $N_{ccc}(\vw)$, the smallest length $n$ such that $A_4(n',2w-1,\vw)=\lfloor\dfrac{n'}{w_1}\rfloor$ for all $n'\geq n$, with only four possible exceptions. The exact values of $N_{ccc}(\vw)$ show that our newly established lower bound of $N_{ccc}(\vw)$ is tight in these cases. Our main result, Main Theorem~\ref{main} follows from Propositions~\ref{boundccc}, \ref{nccc3}, \ref{nccc4} and the refining method in Lemma~\ref{refine1}.

\bibliographystyle{IEEEtran}



\end{document}